\documentclass[11pt]{amsart}

\usepackage{amscd,amssymb,amsthm,amsmath,amssymb,mathrsfs,enumerate,enumitem,hyperref}
\usepackage[matrix,arrow,curve]{xy}
\usepackage[margin=2cm]{geometry}
\usepackage[matrix,arrow,curve]{xy}

\allowdisplaybreaks
\numberwithin{equation}{section}

\theoremstyle{plain}
\newtheorem{theorem}{Theorem}[section]

\newtheorem{lemma}[theorem]{Lemma}
\newtheorem{corollary}[theorem]{Corollary}
\theoremstyle{definition}
\newtheorem{definition}[theorem]{Definition}
\theoremstyle{remark}
\newtheorem{remark}[theorem]{Remark}

\theoremstyle{definition}
\newtheorem{example}[theorem]{Example}

\DeclareMathOperator{\Pic}{Pic}
\DeclareMathOperator{\NS}{NS}
\DeclareMathOperator{\Jac}{Jac}
\DeclareMathOperator{\MW}{MW}
\DeclareMathOperator{\Bl}{Bl}
\DeclareMathOperator{\Gr}{Gr}
\DeclareMathOperator{\Gal}{Gal}

\newcommand{\C}{\mathbb C}
\newcommand{\F}{\mathbb F}
\newcommand{\Pp}{\mathbb P}
\newcommand{\Q}{\mathbb Q}
\newcommand{\Z}{\mathbb Z}
\newcommand{\cL}{\mathcal L}
\newcommand{\cO}{\mathcal O}
\newcommand{\cU}{\mathcal U}
\newcommand{\kbar}{\overline{\Bbbk}}

\title[Unirational del Pezzo surfaces of degree one]{Unirational del Pezzo surfaces of degree one}

\author{Ivan Cheltsov}
\address{Ivan Cheltsov\\
University of Edinburgh, Edinburgh, Scotland}
\email{I.Cheltsov@ed.ac.uk}

\author{Konstantin Loginov}
\address{Konstantin Loginov\\
Steklov Mathematical Institute, Moscow, Russia}
\email{loginov@mi-ras.ru}

\author{Dmitri Orlov}
\address{Dmitri Orlov\\
Steklov Mathematical Institute, Moscow, Russia}
\email{orlov@mi-ras.ru}

\author{Yuri Prokhorov}
\address{Yuri Prokhorov\\
Steklov Mathematical Institute, Moscow, Russia}
\email{prokhoro@mi-ras.ru}


\date{}

\begin{document}

\begin{abstract}
We construct explicit unirational del Pezzo surfaces of degree $1$ with
arithmetic Picard rank one over $\mathbb{Q}$, $\mathbb{F}_5$, and $\mathbb{C}(t)$. Moreover, we prove that
every smooth real geometrically rational surface is unirational over $\mathbb{R}$ if and only if it has a~real point.
\end{abstract}

\maketitle

\section{Introduction}
\label{sec:intro}

Let $\Bbbk$ be a~perfect field of characteristic different from $2$ and $3$, and let $S$ be a~smooth del Pezzo surface that is defined over $\Bbbk$.
Then $d=(-K_S)^2$ is said to be the degree of the surface $S$. If $S(\Bbbk)\ne\varnothing$ and~$d\ge 5$, then $S$ is $\Bbbk$-rational. If $d\le 4$ and $S$ is \textit{minimal}, then $S$ is irrational \cite{IskovskikhMinimal}. If $S(\Bbbk)\neq\varnothing$ and $d\in\{3,4\}$, then $S$ is always unirational  \cite{Segre43,KollarCubic}. 

In \cite{SalgadoTestaVarilly}, Salgado, Testa, and V\'arilly-Alvarado extended Manin's result \cite[Theorem~29.4]{ManinCubicForms} to prove that $S$ is also unirational in the~case when $d=2$ and $S(\Bbbk)$ contains a~point away from the~ramification curve of the~anticanonical double cover that is not a~generalized Eckardt point. 

\begin{remark}
\label{remark:dP2-finite}
Suppose that $d=2$ and $\Bbbk$ is a finite field. Then $S$ is unirational \cite{FestiVanLuijk}.
\end{remark}

\begin{remark}
\label{remark:dP2-real-intro}
Suppose that $d=2$, $\Bbbk=\mathbb{R}$ and $S(\mathbb{R})\ne\varnothing$. Then $S$ is unirational by Corollary~\ref{corollary:dP2-real}.
\end{remark}

In this paper, we study the~unirationality of $S$ in the~remaining case $d=1$. So, we suppose that $d=1$.
Then we have $S(\Bbbk)\ne\varnothing$, because the~base locus of the~pencil $|-K_S|$ is a~$\Bbbk$-point.
Moreover, if $\rho(S)\ne 1$, then $S$ is $\Bbbk$-unirational by \cite[Corollary~8]{KollarMella}, where $\rho(S)=\mathrm{rk}\,\mathrm{Pic}(S)$.

\begin{remark}
If $\Bbbk$ is infinite and $S(\Bbbk)$ is not Zariski dense, $S$ is not unirational over $\Bbbk$. But we do not know whether $S(\Bbbk)$ is always Zariski dense when $\Bbbk$ is infinite, while many such (Zariski dense) examples have been constructed in \cite{DesjardinsWinter}.
\end{remark}

In \cite{DesjardinsWinter}, Desjardins and Winter mentioned that no del Pezzo surfaces of degree $1$ with Picard rank one are known to be unirational or non-unirational. The goal of this paper is to partially close this gap. Namely, we provide several examples of unirational del Pezzo surfaces of degree $1$ with Picard rank~one. Unfortunately, we do not know whether there exists a~del Pezzo surface of degree $1$ that is not unirational.

We first present three unirational examples over $\Q$, $\F_5$ and $\C(t)$. Over $\Q$, we set
$$
S_{\Q}=\Big\{y^2=x^3+uv^2(3u-v)x+\frac14uv\bigl(4v^4-23uv^3-18u^2v^2+u^3v-4u^4\bigr)\Big\}\subset \Pp_{\Q}(1_u,1_v,2_x,3_y).
$$
Then $S_{\Q}$ is a~smooth del Pezzo surface of degree $1$. One can show that $\rho(S_{\Q})=1$. We prove

\begin{theorem}
\label{thm:explicit-Q}
There is a~dominant rational map $\Pp^2_{\Q}\dashrightarrow S_{\Q}$ of degree $9$.
\end{theorem}

Let $S_{\F_5}$ be the~reduction of the surface $S_{\Q}$  modulo $5$. Then $S_{\F_5}$ is a~smooth del Pezzo surface of~degree~$1$, which is defined over $\F_5$. Moreover, one can show that $\rho(S_{\F_5})=1$. Similarly, we prove

\begin{theorem}
\label{thm:explicit-F5}
There is a~dominant separable rational map $\Pp^2_{\F_5}\dashrightarrow S_{\F_5}$ of degree $9$.
\end{theorem}

Over the~function field $\C(t)$, we set
$$
S_t=\big\{y^2=x^3-3u^2v^2x-u^5v-uv^5-27t^2u^2v^4\big\}\subset\Pp_{\C(t)}(1_u,1_v,2_x,3_y).
$$
Then $S_t$ is a~smooth del Pezzo surface of degree $1$ defined over $\C(t),$ and $\rho(S_t)=1$. We prove

\begin{theorem}
\label{thm:explicit-Ct}
There is a~dominant rational map $\Pp^2_{\C(t)}\dashrightarrow S_t$ of degree $9$.
\end{theorem}

Over $\mathbb R$, we prove a~stronger result:

\begin{theorem}\label{thm:real-unirationality}
Every smooth real del Pezzo surface of degree $1$ is unirational.  More precisely, it admits a~dominant rational map
from $\Pp^2_{\mathbb R}$ of degree at most $96$.
\end{theorem}

Now, applying \cite[Corollary~4.4]{IskovskikhPencil}, we obtain the~following result.

\begin{corollary}
\label{cor:all-real-rational-surfaces}
Let $X$ be a~smooth geometrically rational real surface. Then $X$ is unirational if and only if   $X(\mathbb R)\neq\varnothing$.
\end{corollary}

Arguing as in the~proof of Theorem~\ref{thm:explicit-Q}, we can produce infinitely many examples of unirational smooth del Pezzo surfaces of degree $1$ defined over $\mathbb{Q}$ that all have Picard rank one. This gives the~following result.

\begin{corollary}
\label{corollary:moduli-intro}
The~coarse moduli space of smooth complex del Pezzo surfaces of degree $1$ contains a~Zariski-dense (countable) subset whose points
parametrize del Pezzo surfaces of degree $1$ defined over $\mathbb{Q}$ that are $\Q$-unirational and have arithmetic Picard rank one.
\end{corollary}

The idea of our unirationality construction (presented in Section~\ref{sec:construction}) is not new --- it has already been used by Dolgachev and Gross in \cite{DolgachevGross}. Let us briefly describe it. Let $Z$ be a~smooth del Pezzo surface defined over~$\Bbbk$ such that $(-K_Z)^2\ge 2$, let $\mathcal{P}\subset|-K_Z|$ be a~general pencil, let $\Sigma$ be its base locus, and let~$A\in\Pic(Z)$ be an ample line bundle. Then we have the~following commutative diagram:
$$
\xymatrix{
&Y\ar@{->}[dl]_{\pi}\ar@{->}[dr]^{\theta}\ar@{-->}[rr]^{\mu_A} && J\ar@{->}[dr]^{\varpi}\ar@{->}[dl]_{\vartheta}&\\
Z\ar@{-->}[rr]_{\phi} &&\mathbb{P}^1_{\Bbbk}&& S\ar@{-->}[ll]^{\varphi}}
$$
where
\begin{itemize}
\item $\pi$ is the~blow up of $\Sigma$,
\item $\theta$ is a~genus-one fibration given by $|-K_Y|$,
\item $\phi$ is the~rational map given by $\mathcal{P}$,
\item $S$ is a~smooth del Pezzo surface of degree $1$,
\item $\varphi$ is the~rational map given by $|-K_S|$,
\item $\varpi$ is the~blow up of the~base locus of $|-K_S|$,
\item $\mu_A$ is the~relative Abel--Jacobi map for $\theta$ induced by $\pi^*(A)$,
\item $\vartheta$ is the~Jacobian fibration of $\theta$.
\end{itemize}
Moreover, $\mu_A$ is dominant and has degree $n^2$ for $n=A\cdot(-K_Z)>0$.
If $\Sigma$ is $\Bbbk$-irreducible, then~$\rho(S)=\rho(Z)$.
To prove Theorems~\ref{thm:explicit-Q}, \ref{thm:explicit-F5}, \ref{thm:explicit-Ct}, we let $Z=\Pp^2$ and $A=\mathcal{O}_{\mathbb{P}^2}(1)$, so $n=3$.
Similarly, to prove Theorem~\ref{thm:real-unirationality}, we let $Z$ be a~real del Pezzo surface of degree $2$ and Picard rank $1$, which proves unirationality of all real del Pezzo surfaces of degree $1$. In Example~\ref{example:DP1-bad}, we show that this approach cannot be used to prove unirationality of all smooth del Pezzo surfaces of degree $1$ defined over $\mathbb{Q}$.

\subsection*{Structure of the~paper}
In Section~\ref{sec:elliptic-prelim}, we present some known results about genus-one fibrations, and their relative Jacobians.
In Section~\ref{sec:construction}, we describe our unirationality construction in detail. In Section~\ref{sec:examples}, we prove Theorems~\ref{thm:explicit-Q}, \ref{thm:explicit-F5}, \ref{thm:explicit-Ct} and Corollary~\ref{corollary:moduli-intro}. In Section~\ref{sec:real-case}, we prove Theorem~\ref{thm:real-unirationality}. 
In Appendix~\ref{sec:prokhorov-construction}, we briefly discuss unirationality for del Pezzo surfaces of degree $\ge 2$.

\subsection*{AI Disclosure}
This project arose from a conversation at the Steklov Mathematical Institute
about mathematical problems that might be approached with the aid of artificial intelligence.
We tested one idea using Chat GPT-5.6~Sol by OpenAI, which led to our examples of unirational smooth del Pezzo surfaces of degree $1$ with Picard rank~one. We~have checked all arguments and computations and take full responsibility for the~contents of the~paper.

\subsection*{Acknowledgments}
We thank Sergey Gorchinskiy and Constantin Shramov for very helpful discussions. Ivan Cheltsov has been supported by Simons Collaboration grant \emph{Moduli of varieties}. The work of K. Loginov, D. Orlov and Yu. Prokhorov was performed at the Steklov International Mathematical Center and supported by the Ministry of Science and Higher Education of the Russian Federation (agreement no. 075-15-2025-303).

\section{Genus-one fibrations}
\label{sec:elliptic-prelim}

As in Section~\ref{sec:intro}, we let $\Bbbk$ be a~perfect field of characteristic different from $2$ and $3$. Recall that a~\emph{genus-one fibration} over $\Pp^1$ is a~proper flat morphism $\theta\colon Y\to\Pp^1$ such that
\begin{itemize}
\item $Y$ is a~smooth projective surface defined over $\Bbbk$,
\item $\theta$ has geometrically connected fibers,
\item the~generic fiber of $\theta$ is a~smooth geometrically irreducible curve of genus $1$.
\end{itemize}
An \emph{elliptic fibration} is a~genus-one fibration together with a~section.

Let $\theta\colon Y\to\Pp^1$ be a~relatively minimal genus-one fibration without multiple fibers, let $K=\Bbbk(\Pp^1)$, and let $Y_\eta$ be its generic fiber. Then the~relative Jacobian of the~fibration $\theta$ is the~relatively minimal elliptic fibration
$$
\vartheta\colon J\longrightarrow\Pp^1
$$
such that $J_\eta=\mathrm{Jac}(Y_\eta)$, where $J_\eta$ is the~generic fiber of the~fibration $\vartheta$.
Note that $J_\eta$ is an elliptic curve over $K$, so, in particular, it has a~$K$-point. Hence,  $\vartheta$ admits a~section, which we will denote by $C$.

\begin{example}
\label{example:Fisher}
Suppose that $Y_\eta$ is a~plane cubic curve. Let $g\in K[x,y,z]$ be a cubic form such that 
$$
Y_\eta=\{g(x,y,z)=0\}\subset\mathbb{P}^2_{K}.
$$
Using \cite{FisherInv}, we can find $J_\eta$ as follows.
Let $c_4(g)$ and $c_6(g)$ be Fisher's invariants of degrees $4$ and $6$, respectively, with the~normalization as in \cite[Section~7.2]{FisherInv}.
Then, by \cite[Proposition~2.3]{FisherInv}, we get
\begin{equation}
\label{eq:fisher-jacobian}
J_\eta=\big\{zy^2=x^3-27c_4(g)xz^2-54c_6(g)z^3\big\}\subset\mathbb{P}^2_{K}.
\end{equation}
Recall from \cite[Theorem~8.5]{FisherHess} that $c_4(g)$ and $c_6(g)$ are determined by the~identity
\begin{equation}
\label{eq:fisher-identity}
H\big(g+\mu H(g))=3\bigl(c_4(g)\mu+2c_6(g)\mu^2+c_4^2(g)\mu^3\bigr)g+\bigl(1-3c_4(g)\mu^2-2c_6(g)\mu^3\bigr)H(g),
\end{equation}
where $H(g)=-\frac{1}{2}\det(\mathrm{Hess}(g))$. 
\end{example}

Put $\cL=(R^1\vartheta_*\cO_J)^\vee$. Then it follows from \cite{Miranda} that the~minimal Weierstrass model of $\vartheta$ is
$$
\big\{y^2z=x^3+a_4xz^2+a_6z^3\big\}\subset \Pp=\Pp_{\Pp^1}(\cO_{\Pp^1}\oplus\cL^2\oplus\cL^3)
$$
for some sections $a_4\in H^0(\Pp^1,\cL^4)$ and $a_6\in H^0(\Pp^1,\cL^6)$, where $z\in H^0(\cO_{\Pp}(1))$, $x\in H^0(\cO_{\Pp}(1)\otimes p^*(\cL^2))$,
$y\in H^0(\cO_{\Pp}(1)\otimes p^*(\cL^3))$,
where $p\colon\Pp_{\Pp^1}(\cO_{\Pp^1}\oplus\cL^2\oplus\cL^3)\to\Pp^1$ is the~natural projection.

\begin{lemma}
\label{lemma:general-weierstrass}
Suppose that $J$ is a~geometrically rational surface.
Then $\deg(\cL)=1$. Moreover, the~following conditions are equivalent:
\begin{enumerate}[label=\rm(\roman*)]
\item every geometric fiber of $\vartheta$ is irreducible,
\item the~minimal Weierstrass model of $\vartheta$ is smooth.
\end{enumerate}
\end{lemma}

\begin{proof}
Since $\vartheta$ has no multiple fibers, the~canonical bundle
formula gives
$$
K_J\sim\vartheta^*(K_{\Pp^1}\otimes\cL),
$$
and $\deg(\cL)=\chi(\cO_J)=1$. the~minimal resolution of the~Weierstrass model
contracts exactly the~components of fibers of $\vartheta$ that are disjoint from $C$, which easily implies the~remaining assertions.
\end{proof}

Now, we assume that $J$ is geometrically rational. Then $\deg(\cL)=1$. Let $F$ be a~fiber of~$\vartheta$.  Then 
$$
K_J\sim -F,
$$
so $C$ is a~$(-1)$-curve by the~adjunction formula. Hence, there is a~birational morphism $\varpi\colon J\to S$ such that $S$ is a~smooth weak del Pezzo surface of degree $1$, and $\varpi(C)$ is the~base locus of the~pencil $|-K_S|$. Moreover, the~following conditions are equivalent:
\begin{enumerate}[label=\rm(\roman*)]
\item every geometric fiber of $\vartheta$ is irreducible,
\item $S$ is a~del Pezzo surface of degree $1$, i.e., $-K_S$ is ample.
\end{enumerate}
If $-K_S$ is ample, the~equation of $S$ can be derived from the~Weierstrass equation as follows:
\begin{equation}
\label{eq:weighted-model}
\big\{y^2=x^3+a_4(u,v)x+a_6(u,v)\big\}\subset\Pp_{\Bbbk}(1_u,1_v,2_x,3_y),
\end{equation}
where $u$ and $v$ form a~basis of $H^0(\Pp^1,\cL)$, and $a_4$ and $a_6$ are polynomials of degree $4$ and $6$, respectively.
If~the divisor $-K_S$ is not ample, then \eqref{eq:weighted-model} is the~equation of the~anticanonical model of $S$, which is a~singular del Pezzo surface of degree $1$ that has Du Val singularities.

Now, we let $\MW(J_{\kbar}/\mathbb{P}^1_{\kbar})$ be the~group of sections of $\vartheta_{\kbar}$, and let
$T(J)\subseteq\NS(J_{\kbar})$ be the~sublattice generated by $C$,
the fiber $F$, and the~components of reducible geometric fibers of $\vartheta$ that do
not meet $C$. Then we have the~following exact sequence:
$$
0\longrightarrow T(J)\longrightarrow\NS(J_{\kbar})
 \xrightarrow{\operatorname{res}}
 \MW(J_{\kbar}/\mathbb{P}^1_{\kbar})\longrightarrow0,
$$
where the~last map is the restriction to the~generic fiber, with its degree
removed by a~multiple of $C$.  Explicitly, it is induced by
$$
D\longmapsto\bigl[D\big|_{\mathcal{J}}-(D\cdot F)C\big|_{\mathcal{J}}\bigr],
$$
where $\mathcal{J}$ is the~generic fiber of $\overline{\vartheta}\colon J_{\kbar}\to \mathbb{P}^1_{\kbar}$.
Hence, by \cite[Theorem~1.3]{Shioda} or \cite[Theorem~6.3]{SchuettShioda}, we get
$$
 \rho(J_{\kbar})=2+\operatorname{rk}\MW(J_{\kbar}/\mathbb{P}^1_{\kbar})
 +\sum_{v\in \mathbb{P}^1_{\kbar}}(m_v-1),
$$
where $m_v$ is the~number of irreducible components of the~geometric fiber of $\vartheta$
over $v$.

\begin{lemma}
\label{prop:shioda-tate}
Suppose that every geometric fiber of $\vartheta$ is irreducible.
Then restriction to the~generic fiber induces canonical
$\Gal(\kbar/\Bbbk)$-equivariant isomorphisms
$$
T(J)^\perp\simeq F^\perp/\Z[F]\simeq\MW(J_{\kbar}/\Pp^1_{\kbar}),
$$
where the~orthogonal complements are taken in $\NS(J_{\kbar})$.
In particular, $\MW(J_{\kbar}/\mathbb{P}^1_{\kbar})$ is torsion-free.
\end{lemma}

\begin{proof}
Follows from \cite[Theorem~1.3 and Lemmas 4.2–4.3]{Shioda} and \cite[Theorem~6.3]{SchuettShioda}.
\end{proof}

\section{The construction}
\label{sec:construction}

Let $\Bbbk$ be a perfect field of characteristic different from $2$ and $3$. Let $Z$ be a~smooth del Pezzo surface defined over $\Bbbk$ of degree $d=(-K_Z)^2\ge 2$, let $\mathcal{P}\subset|-K_Z|$ be a~pencil, and let $\mathcal{P}_{\kbar}$ be its geometric model defined over $\kbar$.

\begin{definition}
The pencil $\mathcal{P}$ is \emph{admissible} if its base locus is reduced, its generic member is smooth, and every (geometric) member of the~pencil $\mathcal{P}_{\kbar}$ is irreducible.
\end{definition}

\begin{remark}
General pencil in $|-K_Z|$ is admissible.
\end{remark}

Suppose that $\mathcal{P}$ is admissible. Let $\Sigma$ be the~base locus of $\mathcal{P}$ and let $\pi\colon Y\to Z$ be the~blow up of $\Sigma$.  Then we have the~following commutative diagram:
$$
\xymatrix{
&Y\ar@{->}[dl]_{\pi}\ar@{->}[dr]^{\theta}&\\
Z\ar@{-->}[rr]_{\phi} &&\mathbb{P}^1_{\Bbbk}}
$$
where $\theta$ is a~genus-one fibration given by $|-K_Y|$, and $\phi$ is the~rational map given by the~pencil~$\mathcal{P}$. Moreover, since  $\mathcal{P}$ is admissible, every geometric fiber of $\theta$ is irreducible.

Let $\vartheta\colon J\to\Pp^1$ be the~relatively minimal Jacobian of the~genus-one fibration $\theta$, and let $C$ be its zero section. Then $J$ is geometrically rational. 

\begin{lemma}
\label{lemma:weierstrass-criterion}
Every geometric fiber of $\vartheta$ is irreducible.
\end{lemma}

\begin{proof}
Since $\theta$ has sections over $\kbar$, geometric fibers of fibrations $\theta$ and $\vartheta$ are isomorphic over $\kbar$, so the~assertion follows.
\end{proof}

So, by Lemma~\ref{lemma:weierstrass-criterion}, the~minimal Weierstrass model of $J$ is smooth. But, as we explained in Section~\ref{sec:elliptic-prelim}, the~section $C$ is a~$(-1)$-curve, and there exists a~birational morphism $\varpi\colon J\to S$ such that $S$ is a~smooth del Pezzo surface of degree $1$, and $\varpi(C)$ is the~base point of the~pencil $|-K_S|$. Thus, we have the~following commutative diagram:
$$
\xymatrix{
&J\ar@{->}[dr]^{\varpi}\ar@{->}[dl]_{\vartheta}&\\
\mathbb{P}^1_{\Bbbk} &&S\ar@{-->}[ll]^{\varphi}}
$$
where $\varphi$ is the~rational map given by the~pencil $|-K_S|$.

Now, we fix an ample divisor $A\in\Pic(Z)$ and we let $n=A\cdot(-K_Z)>0$.

\begin{lemma}
\label{lemma:AJ}
The divisor $\pi^*(A)$ induces a~dominant rational map $\mu_A\colon Z\dashrightarrow S$ of degree $n^2$ that makes the~following diagram commute:
$$
\xymatrix{
Y\ar@{->}[d]_{\pi}\ar@{->}[drr]^{\theta}\ar@{-->}[rrrr]^{\mu_A} &&&& J\ar@{->}[d]^{\varpi}\ar@{->}[dll]_{\vartheta}\\
Z\ar@{-->}[rr]_{\phi} &&\mathbb{P}^1_{\Bbbk}&& S\ar@{-->}[ll]^{\varphi}}
$$
\end{lemma}

\begin{proof}
Let $Y_{\eta}$ be the~generic fiber of $\theta$.
Then the~restriction of $\pi^*(A)$ to the~curve $Y_{\eta}$ has degree $n$,
and the~Abel--Jacobi morphism $Y_{\eta}\to\Jac(Y_{\eta})$ is given by
$$
P\longmapsto\Big[\cO_{Y_{\eta}}(nP)\otimes \big(\pi^*(A)\big|_{Y_{\eta}}\big)^{-1}\Big].
$$
This map is defined over $\Bbbk(\Pp^1)$.
By construction, it has degree $n^2$ and  is dominant, so it determines the~required rational dominant map $Y\dashrightarrow J$, so the assertion follows.
\end{proof}

We say that the~map $\mu_A$ in Lemma~\ref{lemma:AJ} is the~relative Abel--Jacobi map for $\theta$ induced by $\pi^*(A)$.

\begin{example}
\label{example:maps-and-unirationality}
Let $Z=\Pp^2$  and let $A=\mathcal{O}_{\mathbb{P}^2}(1)$.
Then $\mu_A$ induces a~dominant rational map $\Pp^2\dashrightarrow S$, which has degree $9$.
In particular, $S$ is unirational over $\Bbbk$.
\end{example}

Let $r(\Sigma)$ be the~number of $\Gal(\kbar/\Bbbk)$-orbits in $\Sigma$. Then $\rho(Y)=\rho(Z)+r(\Sigma)$.

\begin{lemma}
\label{lemma:picard-ranks}
Let $F_Y$ and $F_J$ be the~fiber classes of $\theta_{\kbar}$ and
$\vartheta_{\kbar}$, respectively, and put
$$
\Lambda_Y=F_Y^\perp/\Z[F_Y]
\qquad\text{and}\qquad
\Lambda_J=F_J^\perp/\Z[F_J].
$$
Restriction to the~geometric generic fibers induces canonical
$\Gal(\kbar/\Bbbk)$-equivariant isomorphisms
$$
\Lambda_Y\simeq\MW(J_{\kbar}/\Pp^1_{\kbar})\simeq\Lambda_J.
$$
In particular, $\rho(J)=\rho(Y)$.
\end{lemma}

\begin{proof}
Put $K=\Bbbk(\Pp^1)$, and let $Y_\eta$ and $J_\eta$ be the~generic fibers
of $\theta$ and $\vartheta$, respectively. Then
$$
J_\eta=\Jac(Y_\eta).
$$
This is the~Shioda--Tate isomorphism for a~curve and its Jacobian
\cite[4.1.1]{Ulmer}. After base change to $\kbar$, the
exceptional curves of $Y_{\kbar}\to Z_{\kbar}$ are sections. Applying
\cite[4.2.1]{Ulmer} to $Y$ and $J$, gives
$$
\rho(Y)=2+\operatorname{rk}\MW(J_\eta)=\rho(J).
$$
The claim follows.
\end{proof}

\begin{corollary}
\label{corollary:picard-ranks}
One has $\rho(S)=\rho(Z)+r(\Sigma)-1$.
\end{corollary}

\begin{corollary}
\label{corollary:corollary-picard-ranks}
If $\rho(Z)=1$ and $\Sigma$ is irreducible over $\Bbbk$, then $\rho(S)=1$.
\end{corollary}

Let us conclude this section by presenting an example of a smooth del Pezzo surface of degree $1$ defined over $\mathbb{Q}$ that has Picard rank one, but it cannot be obtained using our construction. To do this, we set  
\begin{equation}
\label{eq:KSperp}
\begin{array}{lll}
K_{S_{\kbar}}^\perp&=&\{D\in\Pic(S_{\kbar})\mid D\cdot K_{S_{\kbar}}=0\},
\\[1em]
K_{Z_{\kbar}}^\perp&=&\{D\in\Pic(Z_{\kbar})\mid D\cdot K_{Z_{\kbar}}=0\}.
\end{array}
\end{equation}
Let $B_{\Z}\subset\NS(Y_{\kbar})$ be the~sublattice generated by the~exceptional
curves of the~blow up $\pi_{\kbar}\colon Y_{\kbar}\to Z_{\kbar}$, put~$B_{\mathbb{Q}}=B_{\Z}\otimes\Q$, and let
$$
B^0_{\mathbb{Q}}=\{D\in B_{\mathbb{Q}}\mid D\cdot K_{Y_{\kbar}}=0\}.
$$

\begin{lemma}
\label{lem:obstruction-decomposition}
There exists a~canonical $\Gal(\kbar/\Bbbk)$-equivariant isomorphism
$$
K_{S_{\kbar}}^\perp\otimes \mathbb{Q}\simeq (K_{Z_{\kbar}}^\perp\otimes \mathbb{Q})\oplus B^0_{\mathbb{Q}},
$$
where $\mathrm{dim}(K_{Z_{\kbar}}^\perp\otimes \mathbb{Q})=9-d$ and $\mathrm{dim}(B^0_{\mathbb{Q}})=d-1$.
\end{lemma}

\begin{proof}
Pullback by $\varpi$ identifies $K_{S_{\kbar}}^\perp\otimes \mathbb{Q}$ with
$T(J)^\perp\otimes\Q$. Hence, Lemmas~\ref{prop:shioda-tate} and~\ref{lemma:picard-ranks} give canonical
$\Gal(\kbar/\Bbbk)$-equivariant isomorphisms
$$
K_{S_{\kbar}}^\perp\otimes \mathbb{Q}\simeq\Lambda_J\otimes\Q
\simeq\Lambda_Y\otimes\Q.
$$
The~decomposition
$
\NS(Y_{\kbar})=\pi^*(\NS(Z_{\kbar}))\oplus B_{\Z}
$
induces a~natural injective map
$$
\pi^*(K_{Z_{\kbar}}^\perp\otimes \mathbb{Q})\oplus B^0_{\mathbb{Q}}
\longrightarrow\Lambda_Y\otimes\Q.
$$
Both sides have dimension $8$, so this map is an~isomorphism.
\end{proof}

In particular, if $2\le d\le 8$, then the rational Galois representation
$K_{S_{\kbar}}^\perp\otimes \mathbb{Q}$ is reducible. Using this, one can construct a~smooth del Pezzo surface of degree $1$ with a Picard rank equal to one which cannot be obtained via our construction using unirational $Z$. 

\begin{example}
\label{example:DP1-bad}    
Set
$$
X_{\mathbb{Q}}=\big\{y^2+(u+v)xy=x^3+uvx^2+u^3(u+v)x+uv(u^2+uv+v^2)^2\big\}\subset\Pp_{\Q}(1_u,1_v,2_x,3_y),
$$
and let $X_{\mathbb{F}_2}$ be its reduction modulo $2$.
Then $X_{\mathbb{Q}}$ and $X_{\mathbb{F}_2}$ are smooth del Pezzo surfaces of degree $1$ defined over $\mathbb{Q}$ and $\mathbb{F}_2$, respectively.
As above, let $K_{X_{\overline{\F}_2}}^\perp$ be the~orthogonal complement in $\NS(X_{\overline{\F}_2})\otimes\Q$ to $K_{X_{\overline{\F}_2}}$, and let $\Phi(t)$ be its characteristic polynomial of the~geometric Frobenius on this space. Then 
$$
\Phi(t)=t^8-t^6+t^4-t^2+1.
$$
Since $\Phi(1)=1$, Frobenius has no fixed vector on $K_{X_{\overline{\F}_2}}^\perp$, so $\rho(X_{\F_2})=1$, which also implies that $\rho(X_{\Q})=1$. 
Moreover, since $\Phi(t)$ is irreducible, $K_{X_{\overline{\Q}}}^\perp$ is also irreducible. 
In particular, $X_{\Q}$ cannot be obtained via our construction with $Z\ne\mathbb{P}^2_{\Q}$.
Similarly, if $X_{\Q}\simeq S$ with $Z=\mathbb{P}^2_{\Q}$ and $\Bbbk=\Q$, then 
$K_{S_{\kbar}}\otimes \mathbb{Q}^\perp\simeq B^0_{\mathbb{Q}}$,
which gives
$$
\Phi(t)=\frac{t^9-1}{t-1}.
$$
This is a~contradiction.
\end{example}

\section{Examples}
\label{sec:examples}

In this section, we prove Theorems~\ref{thm:explicit-Q}, \ref{thm:explicit-F5}, \ref{thm:explicit-Ct} and Corollary~\ref{corollary:moduli-intro}.
As above, we use the notation and assumptions of Section~\ref{sec:construction}. We set $Z=\mathbb{P}^2_{x,y,z}$ and let $A=\mathcal{O}_{\mathbb{P}^2}(1)$ as in Example~\ref{example:maps-and-unirationality}.

\subsection{Rational numbers and $\mathbb{F}_5$}
\label{sec:arithmetic-model}

Let $\Bbbk=\Q$ or $\Bbbk=\mathbb{F}_5$. Set $f_0=yz^2-x^3$ and $f_1=xz^2+y^3+y^2z-z^3$.

\begin{lemma}
\label{lemma:Q-pencil}
The pencil $uf_0+vf_1=0$ is admissible and its base subscheme is reduced and irreducible over $\Bbbk$, where $[u:v]\in\mathbb{P}^1$.
\end{lemma}

\begin{proof}
The base subscheme is $\{z=1,y=x^3,x^9+x^6+x-1=0\}$, so it is reduced and irreducible, because the polynomial  $x^9+x^6+x-1$ is irreducible over~$\Bbbk$.  

Over $\overline{\Bbbk}$, every curve in the pencil is irreducible. To show this, we can analyze its singular members.
If~$\Bbbk=\mathbb{Q}$, the pencil has two cuspidal and eight nodal curves defined over $\overline{\mathbb{Q}}$.
Similarly, if $\Bbbk=\mathbb{F}_5$, it has three cuspidal and six nodal members defined over $\overline{\mathbb F}_5$. 
Each of these curves has a unique singular point. Hence, every singular member is geometrically integral, while
the remaining members are smooth. 
\end{proof}

Let $\mathcal{P}$ be the~pencil $uf_0+vf_1=0$, where $[u:v]\in\mathbb{P}^1$.
Let us use Example~\ref{example:Fisher} to find $J_\eta$. We have
$$
Y_\eta=\{f_0+\lambda f_1=0\}\subset\mathbb{P}^2_K,
$$
where $K=\Bbbk(\lambda)$. So, we set $g=f_0+\lambda f_1$. Then, in the~notations of Example~\ref{example:Fisher}, we have
$$
H(g)=4\bigl(9x^2y\lambda^2+3x^2z\lambda^2-3xy^2\lambda^2+9xy^2\lambda-27xyz\lambda^2-3xyz\lambda-9xz^2\lambda^2-3xz^2+3yz^2\lambda^3+z^3\lambda^3\bigr).
$$
Hence, \eqref{eq:fisher-identity} gives 
$c_4(g)=48\lambda^2(\lambda-3)$ and $c_6(g)=-216\lambda(4\lambda^4-23\lambda^3-18\lambda^2+\lambda-4)$, so that \eqref{eq:fisher-jacobian} gives
$$
J_\eta=\Big\{zy^2=x^3+\lambda^2(3-\lambda)xz^2+\frac{1}{4}\lambda\bigl(4\lambda^4-23\lambda^3-18\lambda^2+\lambda-4\bigr)z^3\Big\}\subset\mathbb{P}^2_K.
$$

\begin{proof}[Proofs of Theorems~\ref{thm:explicit-Q} and \ref{thm:explicit-F5}]
Note that $J$ is uniquely determined by $J_\eta$. Thus, if $\Bbbk=\mathbb{Q}$, then $S\simeq S_{\Q}$.
Similarly, if $\Bbbk=\mathbb{F}_5$, then $S\simeq S_{\mathbb{F}_5}$.
By Corollary~\ref{corollary:corollary-picard-ranks}, $\rho(S)=1$, so assertions follow from Lemma~\ref{lemma:AJ}.
\end{proof}

\subsection{Function field}

Now, we let $\Bbbk=\C(t)$. Set $f_0=yz^2-x^3$ and $f_1=y^3-xz^2-2tz^3$.

\begin{lemma}
\label{lemma:Ct-pencil}
The pencil $uf_0+vf_1=0$ is admissible, where $[u:v]\in\mathbb{P}^1$.
\end{lemma}

\begin{proof}
See the proof of Lemma~\ref{lemma:Q-pencil}.
\end{proof}

Thus, we can let $\mathcal{P}$ be the~pencil $uf_0+vf_1=0$, where $[u:v]\in\mathbb{P}^1$.

\begin{lemma}
\label{lemma:Ct-Sigma}
The locus $\Sigma$ is irreducible.
\end{lemma}

\begin{proof}
Since $\Sigma=\{z=1,\ y=x^3,\ x^9-x-2t=0\}$ and $x^9-x-2t$ are irreducible, the~assertion follows.
\end{proof}

Now, let us use Example~\ref{example:Fisher} to find the equation of $J_\eta$. We have
$$
Y_\eta=\{f_0+\lambda f_1=0\}\subset\mathbb{P}^2_K,
$$
where $K=\Bbbk(\lambda)$. So, we set $g=f_0+\lambda f_1$. Then, in the~notations of Example~\ref{example:Fisher}, we have
$$
H(g)=-12\bigl(3\lambda^2x^2y-3\lambda xy^2+18t\lambda^2xyz+xz^2-\lambda^3yz^2\bigr).
$$
Hence, \eqref{eq:fisher-identity} gives
$c_4(g)=144\lambda^2$ and $c_6(g)=864\lambda(\lambda^4+27t^2\lambda^3+1)$, so that \eqref{eq:fisher-jacobian} gives
$$
J_\eta=\big \{zy^2=x^3-3\lambda^2xz^2-(\lambda+\lambda^5+27t^2\lambda^4)z^3\big \}\subset\mathbb{P}^2_K.
$$

\begin{proof}[Proof of Theorem~\ref{thm:explicit-Ct}]
Since $J$ is uniquely determined by $J_\eta$, we got $S\simeq S_t$. By Corollary~\ref{corollary:corollary-picard-ranks}, we get~$\rho(S)=1$. Now, the~assertion follows from Lemma~\ref{lemma:AJ}.
\end{proof}

\begin{remark}
\label{remark:monodromy}
Let $\cU\subseteq\Gr\bigl(2,H^0(\Pp^2,\cO_{\Pp^2}(3))\bigr)$ be the~locus parametrizing admissible pencils in $|-K_{\mathbb{P}^2}|$, and let
$\mathcal B\to\cU$ be the~corresponding universal finite \'{e}tale base scheme of degree $9$.
Then it follows from the~proof of Lemma~\ref{lemma:Ct-Sigma} that $\cU\ne\varnothing$. One can show that the~geometric monodromy group of $\mathcal B\to\cU$ is~$\mathfrak{S}_9$.
\end{remark}

\subsection{Moduli}
\label{subsec:moduli}

Let $V=H^0(\Pp^2,\cO_{\Pp^2}(3))\simeq\mathbb{C}^{10}$, and let $\cU\subseteq\Gr(2,V)$ be an~open subset parametrizing admissible pencils in $|-K_{\Pp^2}|$. Then $\operatorname{PGL}_3$ naturally acts $\cU$,
and it follows from \cite[Theorem~7.1(ii)]{MirandaStability} that every pencil in the subset $\cU$ is GIT-stable. 
Blowing up the~nine base points of a~pencil in $\cU$ gives a~rational elliptic surface whose fibers are the~strict transforms of curves in the~pencil. Now, applying the~construction from Section~\ref{sec:construction}
to the~universal plane cubic pencil gives a~$\operatorname{PGL}_3$-invariant dominant morphism $\widetilde\Phi\colon\cU\longrightarrow M_1$, where $M_1$ is the~moduli space of del Pezzo surfaces of degree $1$ \cite{Naruki}. 

Since every point of $\cU$ is GIT-stable, we obtain a quasi-projective geometric quotient $\mathcal{M}=\cU/\!/\operatorname{PGL}_3$,
and $\widetilde\Phi|_{\cU}$ descends to a dominant and generically finite morphism $\Phi\colon\mathcal{M}\to M_1$. Note that $\Phi$ is not birational, since projectively inequivalent pencils of cubics can lead to isomorphic del Pezzo surfaces. 

\begin{proof}[Proof of Corollary~\ref{corollary:moduli-intro}]
Let $T\subset\cU(\Q)$ be the subset parametrizing admissible pencils defined over $\mathbb{Q}$ whose base locus is irreducible over $\mathbb{Q}$. We know from Lemma~\ref{lemma:Q-pencil} that $T\ne\varnothing$. This implies that is $T$ is Zariski dense in $\cU$. The~dominance of $\widetilde{\Phi}$ therefore shows that $\widetilde{\Phi}(T)$ is Zariski dense in $M_1$.
\end{proof}

\section{The real case}
\label{sec:real-case}

Let $S$ be a~smooth real del Pezzo surface of degree $1$, let $O$ be the~base point of the~pencil $|-K_S|$, let~$\varpi\colon J\to S$ be the~blow up of the~point $O$, and let $C\subset J$ be the~exceptional curve. We have a commutative diagram
$$
\xymatrix{
&J\ar@{->}[dr]^{\varpi}\ar@{->}[dl]_{\vartheta}&\\
\mathbb{P}^1_{\mathbb{R}} &&S\ar@{-->}[ll]^{\varphi}}
$$
where $\varphi$ is the~map given by $|-K_S|$, and $\vartheta$ is an elliptic fibration with section $C$. As in \eqref{eq:KSperp}, set 
$$
K_{S_{\mathbb C}}^\perp=\{D\in\Pic(S_{\mathbb C})\mid D\cdot K_{S_{\mathbb C}}=0\}.
$$
Then, as a lattice equipped with the~intersection form, $K_{S_{\mathbb C}}^\perp\simeq\mathrm{E}_8(-1)$ \cite[Chapter~IV, \S25]{ManinCubicForms}.  

Let $F$ be a fiber of $\vartheta$, and let $T=\Z[C]\oplus\Z[F]$.
Then it follows from Lemma~\ref{prop:shioda-tate} that there exists a natural $\Gal(\mathbb{C}/\mathbb{R})$-equivariant isomorphisms
$$
K_{S_{\mathbb C}}^\perp\simeq T^{\perp}\simeq\MW(J_{\mathbb C}/\Pp^1_{\mathbb C}),
$$
where $\MW(J_{\mathbb C}/\Pp^1_{\mathbb C})$ is the~Mordell--Weil lattice defined in Section~\ref{sec:elliptic-prelim},
which is isomorphic to $\mathrm{E}_8$.

\begin{remark}
\label{remark:roots}
Recall that a~root in $K_{S_{\mathbb C}}^\perp$ is a~class $\alpha\in K_{S_{\mathbb C}}^\perp$ such that $\alpha^2=-2$.  
By \cite[Theorem~23.9]{ManinCubicForms}, there exists a $\Gal(\mathbb{C}/\mathbb{R})$-equivariant bijection between the~roots in $K_{S_{\mathbb C}}^\perp$ 
and $(-1)$-curves on $S_{\mathbb{C}}$, which can be described as follows:
$$
\alpha\longmapsto E_\alpha\sim\alpha-K_S.
$$
If $\alpha$ is a root, and $E_\alpha$ is the~corresponding $(-1)$-curve, its strict transform on $J_{\mathbb C}$ is a~section of the~elliptic fibration $\vartheta_{\mathbb C}\colon J_{\mathbb C}\to\Pp^1_{\mathbb C}$, which is disjoint from the~section $C$. 
\end{remark}

Let $\sigma$ be complex conjugation on $S_{\mathbb C}$, and let us use the~same
notation for its lift to $J_{\mathbb C}$ and for the~induced actions on $K_{S_{\mathbb C}}^\perp$ and $\MW(J_{\mathbb C}/\Pp^1_{\mathbb C})$. 
Then it follows from \cite{CarterWeyl} that 
\begin{itemize}
\item either $\sigma$ acts as $-\mathrm{id}$ on $K_{S_{\mathbb C}}^\perp$,
\item or $\sigma$ fixes a root in $K_{S_{\mathbb C}}^\perp$.
\end{itemize}

\begin{corollary}
\label{corollary:real-involution-dichotomy}
Suppose that $\sigma$ does not act as $-\mathrm{id}$ on $K_{S_{\mathbb C}}^\perp$. Then $S$ is unirational.
More precisely, it admits a~dominant rational map from $\Pp^2_{\mathbb R}$ of degree at most $24$.
\end{corollary}

\begin{proof}
If the~complex conjugation $\sigma$ fixes a root in $K_{S_{\mathbb C}}^\perp$, then $S$ can be obtained by blowing up a real smooth del Pezzo surface of degree $2$ at a real point,  so the assertion follows from Remark~\ref{remark:dP2-real}.
\end{proof}

Now, Theorem~\ref{thm:real-unirationality} is a consequence of the~following result.

\begin{lemma}
\label{lemma:real-minus-identity}
Suppose that $\sigma$ acts as $-\mathrm{id}$ on $K_{S_{\mathbb{C}}}^\perp$. Then there is
a smooth real del Pezzo surface $Z$ of degree $2$ such that
$Z(\mathbb R)\ne\varnothing$ and there exists a~dominant rational map
$Z\dashrightarrow S$ of degree $4$.  
\end{lemma}

\begin{proof}
Let $\alpha\in K_{S_{\mathbb{C}}}^\perp$ be a root, and let $E$ be the~strict transform on $J_{\mathbb{C}}$ of the~corresponding $(-1)$-curve. Then $E$ is a $(-1)$-curve, and $E$ is a~section of the~elliptic fibration $\vartheta_{\mathbb C}\colon J_{\mathbb C}\to\Pp^1_{\mathbb C}$ such that $E\cap C=\varnothing$. By assumption, we have $\sigma(E)=-E$ in the~Mordell--Weil group. On the~other hand, the~translation by $E$ on the~geometric generic fiber of $\vartheta_{\mathbb C}$ extends to a~regular automorphism $t_E\in\mathrm{Aut}(J_{\mathbb C})$, because the~elliptic fibration $\vartheta_{\mathbb{C}}\colon J_{\mathbb C}\to\mathbb{P}^1_{\mathbb{C}}$ is relatively minimal.
Set $\tau=t_E\circ\sigma$. Then 
$$
\tau^2=t_{E+\sigma(E)}=\mathrm{id}.
$$
Therefore, we can twist the real structure on $J_{\mathbb{C}}$ such that $\tau$ is a new complex conjugation. This gives us a~smooth projective real surface $Y$ and a genus-one fibration $\theta\colon Y\to\Pp^1_{\mathbb R}$ such that $Y_{\mathbb C}\simeq J_{\mathbb C}$, and $\vartheta$ is the~relative Jacobian of $\theta$. Since $\tau(C)=E$, the curves $C$ and $E$ deform a~real conjugate pair of disjoint $(-1)$-curves on $Y_{\mathbb C}$.  Hence, we have the~following commutative diagram:
$$
\xymatrix{
&Y\ar@{->}[dl]_{\pi}\ar@{->}[dr]^{\theta}&\\
Z\ar@{-->}[rr]_{\phi} &&\mathbb{P}^1_{\mathbb{R}}}
$$
where $Z$ is a smooth real del Pezzo surface of degree  $(-K_Z)^2=2$, the~morphism $\pi$ is a birational contraction of the~curves $E$ and $C$ to two complex conjugated points of the~surface $Z$, and $\phi$ is the~rational map given by an admissible pencil $\mathcal{P}\subset |-K_Z|$. 
Hence, applying Lemma~\ref{lemma:AJ} to $Z$ with $A=-K_Z$ gives a~rational dominant map $\mu_A\colon Y\dasharrow J$ of degree $4$,
which gives the~required map $Z\dashrightarrow S$.

However, we did not yet prove that $Z(\mathbb R)\ne\varnothing$. To show this, observe that $\sigma$ acts on $\Pic(J_{\mathbb C})\otimes\Q$ 
with trace $-6$. Hence, the~Lefschetz fixed-point formula gives
$$
 \chi(J(\mathbb R))
 =2+\operatorname{tr}\bigl(\sigma\mid
 H^2(J_{\mathbb C}, \mathbb Q)\bigr)=
 2-\operatorname{tr}\bigl(\sigma\mid
 \Pic(J_{\mathbb C})\otimes\Q\bigr)=
 8.
$$
Thus, the elliptic fibration $\vartheta$ has a~real singular fiber, since otherwise $J(\mathbb R)$ would have Euler characteristic zero.  
Since all geometric fibers of $\vartheta$ are irreducible, this real singular fiber has a unique singular point~$\mathfrak{p}$.
Note that both $\sigma$ and $t_E$ must fix $\mathfrak{p}$. 
Then $\tau(\mathfrak{p})=\mathfrak{p}$, so $\mathfrak{p}\in Y(\mathbb{R})$ and, therefore, $\pi(\mathfrak{p})\in Z(\mathbb R)$.
\end{proof}

\begin{proof}[Proof of Theorem~\ref{thm:real-unirationality}]
The assertion follows from Lemma~\ref{lemma:real-minus-identity}, Corollary~\ref{corollary:real-involution-dichotomy} and Remark~\ref{remark:dP2-real}.
\end{proof}

\appendix

\section{Unirationality of del Pezzo surfaces of degree at least two}
\label{sec:prokhorov-construction}

Let $\Bbbk$ be an infinite perfect field of characteristic different from $2$,
and let $X$ be a~smooth del Pezzo surface of degree $d=(-K_X)^2\ge 2$ that is defined over $\Bbbk$.
The goal of this appendix is to give a~simple birational proof of the~following result, which follows from \cite{SalgadoTestaVarilly,KollarCubic}.

\begin{theorem}
\label{thm:prokhorov-unirationality}
Suppose that there is a~point $P\in X(\Bbbk)$ such that the following holds:
\begin{itemize}
\item
if $d\le 3$, then  $P$  does not lie on any exceptional curve of $X_{\kbar}$,
\item 
if $d=2$, then $P$ does not lie on the~ramification divisor of the~anticanonical double cover $X\to\Pp^2$.
\end{itemize}
Then $X$ is $\Bbbk$-unirational.
\end{theorem}

\begin{corollary}
\label{corollary:dP2-real}
Suppose that $d=2$, $\Bbbk=\mathbb{R}$ and $X(\mathbb{R})\ne\varnothing$. Then $X$ is  unirational.    
\end{corollary}

\begin{proof}
Let $\Gamma$ be the~union of the~ramification curve of the~anticanonical double cover $X\to\mathbb{P}^2_{\mathbb{R}}$ and all exceptional curves in $X_{\mathbb C}$.
Then $\Gamma$ is a curve defined over $\mathbb R$.  
On the~other hand, $X(\mathbb R)$ is a~nonempty smooth two-manifold, whereas $\Gamma\cap X(\mathbb R)$ has real dimension at most $1$.
Hence, there is a~point $O\in X(\mathbb R)\setminus\Gamma$, so $X$ is $\mathbb{R}$-unirational by Theorem~\ref{thm:prokhorov-unirationality}.
\end{proof}

\begin{remark}
\label{remark:dP2-real}
If $d=2$, $\Bbbk=\mathbb{R}$ and $X(\mathbb{R})\ne\varnothing$, then it follows from \cite[Chapter~IV, Theorem~29.4]{ManinCubicForms} that there exists a~dominant rational map $\Pp^2_{\mathbb R}\dashrightarrow X$ of degree at most $24$.
\end{remark}

\begin{remark}
For $d=3$, the~condition that the~$\Bbbk$-rational point avoid the~geometric exceptional curves in Theorem~\ref{thm:prokhorov-unirationality} is unnecessary: see \cite[Chapter~IV, Theorem~30.1]{ManinCubicForms} and \cite[Theorem~1.1]{KollarCubic}. Similarly, for $d=2$, it follows from \cite[Corollary~3.3]{SalgadoTestaVarilly} that it is enough to assume that the~$\Bbbk$-rational point is not contained in the~ramification divisor of the anticanonical double cover, and it is not a~generalized Eckardt point, i.e., it is not the intersection point of four $(-1)$-curves.  
\end{remark}

In the remaining part of the appendix, we will prove Theorem~\ref{thm:prokhorov-unirationality}. We suppose that $X(\Bbbk)$ contains a~point $P$ that satisfies all conditions of  Theorem~\ref{thm:prokhorov-unirationality}. Let us show that the surface $X$ is unirational over~$\Bbbk$. We split the proof in several steps.

\subsection*{Birational extremal contractions.}
Suppose first that $X$ admits a~birational extremal contraction over $\Bbbk$.
Its  target is a~del Pezzo
surface of larger degree.  If a~geometric exceptional curve $C'$ on the~target passes through $m\geq1$ centers of the~contraction, then its strict
transform $\widetilde C'$ on $X$ satisfies
$-K_X\cdot\widetilde C'=1-m\leq0$, contrary to the~ampleness of $-K_X$.
Thus, every exceptional curve through the~image of~$P$ avoids the~centers,
and its strict transform is an exceptional curve through $P$, which is
impossible by hypothesis.
Repeating these contractions, we may assume that $X$ is minimal over $\Bbbk$.

\subsection*{Two conic bundles.}
Suppose that $\rho(X)>1$.  Since $X$ is minimal, every extremal contraction
is of fiber type.  Such a~contraction has relative Picard rank one and its
base has Picard rank one, so $\rho(X)=2$. The~two boundary rays therefore
give conic bundles
$$
 f_i\colon X\longrightarrow B_i,
 \qquad i=1,2.
$$
Each $B_i$ is a~smooth curve of genus zero containing $f_i(P)$, and, hence,
$B_i\simeq\Pp^1$.  Let~$C$ be the~fiber of~$f_2$ passing through~$P$.
A singular geometric fiber of a~conic bundle is a~union of two geometric
exceptional curves. Note that on a del Pezzo surface of degree $\ge 4$ 
there are at most two exceptional curves passing through it.
Thus, we may assume that $C$ is smooth, and $P\in C(\Bbbk)$ gives
$C\simeq\Pp^1$.  The~restriction $f_1|_C$ is nonconstant because the~two
extremal rays are distinct.  After the~base change $C\to B_1$, the~conic
bundle $f_1$ has a~section and its total space is rational.  This rational
surface dominates $X$, so $X$ is unirational.

\subsection*{The case when $d\ge 4$.}
We may now assume that $\rho(X)=1$.  If $d\geq5$, then $X$ is rational
by \cite[Chapter~IV, Theorem~29.4]{ManinCubicForms}.  Suppose that $d=4$,
and let
$$
 \sigma\colon W=\Bl_PX\longrightarrow X
$$
be the~blow-up, with exceptional curve $E$. The~standard blow-up criterion
for del Pezzo surfaces shows that $W$ is del Pezzo of degree $3$ because
$P$ lies on no geometric exceptional curve
\cite[Chapter~IV, \S24]{ManinCubicForms}. Thus, $W$ is a~smooth cubic
surface and $E$ is a~line defined over $\Bbbk$.  Projection from $E$
defines a~conic bundle $W\to\Pp^1$ so that $E$ is a~rational bisection.
Base change by this bisection gives a~rational surface dominating $W$, and,
hence, $X$ is unirational.

\subsection*{The case when $d=2$ or $d=3$.}
It remains to consider the~cases $d=3$ and $d=2$ with $\rho(X)=1$. The~surface $W=\Bl_PX$ is del
Pezzo of degree $d-1$ by the~same blow-up criterion.  For $d=2$, this is
where the~additional assumption that $P$ lie outside the~ramification
divisor is used.  Let $\tau: W\to W$ be the~Geiser involution  when
$d=3$ and the~Bertini involution when $d=2$.  These involutions are defined
over $\Bbbk$, and
\begin{equation}
\label{eq:LEtauE}
 E+\tau(E)\sim
 \begin{cases}
  -K_W,  & d=3,\\
  -2K_W, & d=2.
 \end{cases}
\end{equation}
The two birational morphisms $\sigma$ and $\sigma\circ\tau$ form the~corresponding 
symmetric Sarkisov link centered at~$P$.  
Let $E'$ be the~exceptional divisor of the~contraction $\sigma\circ \tau$
and let $D$ be its image under $\sigma$. Thus, $D=\sigma\bigl(\tau(E)\bigr)$.
Obviously, $E'\simeq \mathbb{P}^1$.
By \eqref{eq:LEtauE} we have $\tau(E)\neq E$.
Hence, $D$ is a  rational curve.  Again by \eqref{eq:LEtauE}, one has
$$
 \tau(E)\sim
 \begin{cases}
  \sigma^*(-K_X)-2E,  & d=3,\\
  \sigma^*(-2K_X)-3E, & d=2.
 \end{cases}
$$
Consequently,
$$
 D\in
 \begin{cases}
 |-K_X| & \text{with}\  \operatorname{mult}_P(D)=2
 \quad\text{if }d=3,
\\
 |-2K_X|&  \text{with}\   \operatorname{mult}_P(D)=3
 \quad\text{if }d=2.
 \end{cases}
$$
Now we consider a relative version of the~above construction.
Choose a~nonempty open subset $U\subseteq D_{\mathrm{sm}}$ disjoint from all
geometric exceptional curves and, when $d=2$, from the~ramification
divisor. The~curve $U$ is rational.  In the~constant family
$\mathcal X=X\times U$, let $\Delta$ be the~graph of the~inclusion
$U\hookrightarrow X$, and blow up this section:
$$
 \beta\colon\mathcal W=\Bl_\Delta\mathcal X\longrightarrow\mathcal X.
$$
For every geometric point $q\in U$, the~fiber $\mathcal W_q$ is del Pezzo
of degree $d-1$.  After shrinking
$U$, the~relative anticanonical morphism given by $|-(4-d)K_{\mathcal W}|$ in the~first case and the~relative
bianticanonical morphism in the~second case are finite flat double covers.
Their deck transformations therefore define an involution
$\tau_U\colon\mathcal W\to\mathcal W$ over $U$.  If $\mathcal E$ is the~exceptional divisor of $\beta$, then
$\mathcal E'=\tau_U(\mathcal E)$ is isomorphic to $\mathcal E$.  It is therefore rational, since
$\mathcal E$ is a~$\Pp^1$-bundle over the~rational curve $U$.
Consider the~morphism
$$
 g=\operatorname{pr}_X\circ\beta|_{\mathcal E'}
 \colon\mathcal E'\longrightarrow X.
$$
For every geometric point $q\in U$, the~image of the~fiber
$\mathcal E'_q$ is a~rational curve $D_q\subset X$. The~fiberwise divisor
relation gives
$$
 \operatorname{mult}_q(D_q)=
 \begin{cases}
  2, & d=3,\\
  3, & d=2.
 \end{cases}
$$
In particular, $q\in D_q$ and $D_q$ is singular at $q$.  Hence, the~closure
of $g(\mathcal E')$ contains $D$.
If $g$ were not dominant, this closure would equal $D$.  For general
$q\in U$, this would give $D_q=D$, which is impossible because $D$ is
smooth at $q$ whereas $D_q$ is singular there.  Thus, $g$ is dominant.
Since $\mathcal E'$ is rational, the~surface $X$ is unirational.
This completes the proof of Theorem~\ref{thm:prokhorov-unirationality}.

\end{document}